\documentclass[a4paper,11pt]{article}
\usepackage[left=3cm,top=3cm,right=3cm,nohead,foot=1cm]{geometry}

\usepackage[plainpages=false]{hyperref}
\usepackage{amsfonts,latexsym,rawfonts,amsmath,amssymb,amsthm,mathrsfs}
\usepackage{amsmath,amssymb,amsfonts,latexsym,lscape,rawfonts}

\usepackage[all]{xy}
\usepackage{eufrak}
\usepackage{makeidx}         
\usepackage{graphicx,psfrag}

\usepackage{array,tabularx}

\usepackage{setspace}
\input{xy}
\xyoption{all}

\newtheorem{thm}{Theorem}[section]

\newtheorem{lem}[thm]{Lemma}
\newtheorem{clm}[thm]{Claim}
\newtheorem{prop}[thm]{Proposition}

\theoremstyle{remark}
\newtheorem{rmk}[thm]{Remark}

\theoremstyle{definition}
\newtheorem{Def}[thm]{Definition}
\title{Pseudolocality of Ricci Flow under Integral Bound of Curvature}
\author{Wang Yuanqi\thanks{Department of Mathematics, UW-Madison. Email address: wang56@wisc.edu.}
}
\date{}
\begin{document}
\maketitle
\begin{abstract}We prove a pseudolocality type theorem for compact
Ricci Flow under local  integral bounds of curvature. The main tool
is Local Ricci Flow introduced by Deane Yang in \cite{deane} and
Pseudolocality Theorem of Perelman in \cite{P1}. We also study $L^p$
bounds for the derivatives of curvature and smooth extension of
Local Ricci Flow.
\end{abstract}
\section{Introduction}
In this short paper we are interested in Pseudolocality phenomenon
of compact Ricci Flow. In 1982, R.Hamilton introduced the Ricci flow
in
 \cite{H3}, where he deforms a Riemanniana metric in the direction
 of -2 times its Ricci curvature:  \[ \frac{\partial g_{ij}}{\partial t}=-2Ric_{ij}.\]
 Ricci flow gives a canonical way of deforming an arbitrary
 metric to a critical metric (Einstein metric in particular). This program
 of Ricci flow has been remarkably succesful in the recent years
 thanks to the seminal work of G. Perelman \cite{P1}.\\ \\
 By the nature of this evolution equation,
 any change of initial metric initially and locally, will affect the evolution of metric
 afterwards globally.  However, an elegant and powerful theorem by G. Perelman said that
 if the initial metric is closed to be Euclidean in some unit ball, then, no matter what happen
 far away, this almost Euclidean region will remains to be so for some fixed time.
 This gives us hope as how to control the metric based on local infomation even during a global
 flow. Following Perelman, we want to replace the almost Euclidean isoperimetric constant assumption in Perelman's original theorem by some integral bound of curvature.
 \begin{Def}{$\digamma(n,p,K,\tau,\eta)$ assumption:}\\ For all $n\geq 3, p> \frac{n}{2},\ \eta,\ K$, there exists
$\omega_0$ such that $\forall\ \tau\leq \omega_0$,  We say a ball
$[B(x,r),\ g]$ in a compact Riemannian Manifold satisfy
$\digamma(n,p,K,\tau,\eta)$ assumption if :
\begin{enumerate}
\item $Vol\ B(x,\tau r)\geq w(n)(\eta\tau  r)^n$
\item $|Rm|_{\frac{n}{2}[B(x,\tau r)]}\leq \varpi(n,\ \eta)$
\item $|Ric|_{p[B(x,r)]}\leq Kr^{\frac{n}{p}-2},\ p>\frac{n}{2}$
\end{enumerate}In which $$\varpi(n,\ \eta)=\frac{7}{8}2^{-1-\frac{2}{n}}\frac{(n-2)^2}{n^2(n-1)^2}\cdot\alpha^2(n-1)\cdot\alpha^{\frac{2}{n}-2}(n)\cdot(\eta)^{2n+2}$$ and $\alpha(n)\{\omega(n)\}$ is the volume of n-dim unit
sphere $\{$unit Euclidean ball$\}$.
 \end{Def} The following is our main theorem:
 \begin{thm}{Pseudolocality of Compact Ricci Flow under  Bounds on
Curvature:}\\
\label{T2} For all $n\geq 3, p> \frac{n}{2},\ \eta,K,\ \tau$, there
exists  $\epsilon$ such that:
 \\
If $[g(t),M]$ is a compact solution to Ricci-Flow:
$\partial_tg=-2Ric$, , $0\leq t\leq (\epsilon r)^2$ and when $t=0$
the ball $B_0(x_0,r)$ satisfies $\digamma(n,p,K,\tau,\eta)$
assumption and $R\geq -r^2$ in $B_0(x_0,r)$, then :
\begin{enumerate}
\item $|Rm|(x)\leq\frac{1}{t}+(\epsilon r)^{-2}$
\item $VolB(x,t^{\frac{1}{2}})\geq C(n)t^{\frac{n}{2}}$
\end{enumerate}whenever
$0< t\leq (\epsilon r)^2$ and $d_t(x,x_0)\leq \epsilon r$.
\end{thm}
\begin{rmk}:
\begin{itemize}
\item This theorem is an $\epsilon-$ regularity type theorem for
Ricci Flow. The $\epsilon-$ regularity theorem for Einstein
4-Manifold is well known, one could refer to Theorem 4.4 in
\cite{andersongafa} and Theorem 0.8 in \cite{Tiancheeger} which are
motivations of  our work.
\item The $\digamma(n,p,K,\tau,\eta)$ assumption  ensure the assumption
on isoperimetric constant in Perelman's  Pseudolocality Theorem
\ref{T3} to be valid. Assumption 1 and 2 in
$\digamma(n,p,K,\tau,\eta)$ are definitely indispensable. The
example in section 2 shows that assumption 2 and 3 in
$\digamma(n,p,K,\tau,\eta)$ can not be simultaneously weaken.\\
According to Theorem 0.8 and section 11 in \cite{Tiancheeger}, one
could expect to drop Assumption 1 in $\digamma$ if one works  on the
normalized Ricci Flow: $\partial_tg=-2Ric+\frac{2r}{n}g$. It is also
interesting to investigate on whether Assumption 3 in
$\digamma(n,p,K,\tau,\eta)$ could be dropped.
\item Given a smooth compact Riemannian Manifold M, $\exists \ r_0$ sufficiently small such that the assumptions in $\digamma(n,p,K,\tau,\eta)$
 are all valid in any ball of radius less that $r_0$.

 \item The reason why we need $n\geq 3$ is that we should use Moser
 Iteration to Theorem \ref{T1}. The case $n=2$ should be similar
 with something modified, which we left to the interested reader to
 search on.
 \item The corresponding Finiteness of Diffeomorphism Types and Gromov-Hausdorff compactness of
 those manifolds satisfy the assumption of Theorem \ref{T2} are also
 valid. We will not state them here.
\end{itemize}  \end{rmk}

   The relationship between integral bounds on curvature and Isoperimetric constant has been studied for long. One could refer to the interesting works of
  \cite{gromov}, \cite{croke}, \cite{gallot}, \cite{deane}.

The invention of Local Ricci Flow really enable one to work on this
more conveniently. In \cite{deane}, using proper cutoff function
$\chi$ on a Riemannian Manifold, Deane Yang considered the local
deformation of the metric along Local Ricci Flow:
    \[ \frac{\partial g}{\partial t}=-2\chi^2Ric\]
      The short time existence of LRF is proved in \cite{deane} ,  one could also refer to
\cite{li}. \\
Under LRF, one could see some important  informations on the local
metric structure   under local integral bounds of curvature. Most important of all, LRF is able to control the local integral norm of curvature $\int_{B(x.r)}|Rm|^{\frac{n}{2}}$. Thus it's natural to use LRF to study the pseudolocality phenomenon of Global Ricci Flow.  \\

     In section 4 of  this paper we study the smooth extension of the Local Ricci Flow. Namely, we  provide the details of the  $L^p$ bounds on derivative of
     curvature.
It is unlikely to directly get bounds for
     $\int|\nabla^m
     Rm|^2,\ m\geq 0$ as claimed in Theorem 9.2 in Page 93 of \cite{deane},
     the reason is  illustrated in Theorem \ref{smooth extension of LRF} and it's remark. Our method use induction which is based on $L^p$ estimate. \\
\\
 One might doubt whether  LRF is necessary to  prove Theorem \ref{T2} or not. We think the rescaling approach initiated in \cite{anderson} might  also work to
     get Proposition \ref{strong noncollapsed}, which is crucial
     to Theorem \ref{T2}. However, under the bound of $|Ric|_p,\
     p>\frac{n}{2}$, the best possible regularity for the coordinates are $C^{1,\ \alpha}$ and
     $W^{3,\ p}$. Then one must argue the rescaled limit is flat using
     the assumption that
     $|Rm|_{\frac{n}{2}}$ is sufficient before we take limit. This
     is challenging to me because the limit has little regularity.
     Anyway we think there might be an approach without LRF.
     \\

In a word, in this short paper we mainly prove Theorem \ref{T2} and
study the smooth extension of Local Ricci Flow in section 5 to
further explain Theorem \ref{T1} which is one of the main theorems
in
     \cite{deane}. \\

The arrangement of this paper is: In Section 2 we compute an example
to get a feeling on this problem and see which assumptions are
indispensable. In section 3 we prove Theorem \ref{T2}. In
Section 4 we do $L^p$ estimates and prove Theorem \ref{smooth extension of LRF} on smooth extension of Local Ricci Flow.\\ \\
Acknowlegdement: I would like to thank my advisor Chen Xiuxiong for
introducing me to the interesting area of Pseudolocality of Ricci
Flow and Local Ricci Flow. I would not  be able to finish these
without his useful discussions and encouragement. I  am grateful to
Wang Bing for reading this paper carefully and provide insightful
comments. I also would like to thank Li Hao Zhao for reading  my
paper and Ye Li for reading section 4 carefully.

\section{An Example}
    I would like to point out that the conclusions of Theorem \ref{T2} will not hold
    if
    assumption 2 and 3 in $\digamma(n,p,K,\tau,\eta)$
   are weakened. The counterexample comes from
   S. Angenent and D. Knopf \cite{sigurd}. \\ \\
   More precisely, we have the following statement:
   \begin{clm}In dimension 4, there is a family of Ricci-Flow $g_{G}(t), \ x_G$ , constants $c, K$ independent of $G$
such   that:
   \begin{enumerate}
\item $|Rm|_{2[B_0(x_G,c)]}\leq K$
\item $R\geq 0$ in $B_0(x_G,c)$
\end{enumerate}

 And $g_{G}(t)$ will become singular before time $G$.
 \end{clm}
   Assume $g_0=ds^2+\psi^2(s)dS^n,s\in[-\frac{\pi}{2},\frac{\pi}{2}]$
   could be extended to be a smooth rotationally symmetric
   Riemannian metric on $S^{n+1}$. We define "necks" to be the local
   minimum of $\psi$, "bumps" to the local maximum of $\psi$. ("local" means  the  boundary points $-\frac{\pi}{2}$ and $\frac{\pi}{2}$ are not
   included). It is well know that:\\ \\
   Given a point p, if $\{\frac{e_1}{\psi},...\frac{e_n}{\psi}\}$ is  orthonomal
   at p and is tangent to the $S^n$ fibre, then:
   $$K_{N}\triangleq R(\partial s,\frac{e_i}{\psi},\frac{e_i}{\psi},\partial s)=-\frac{\ddot{\psi}}{\psi},\ K_{T}=R(\frac{e_j}{\psi},\frac{e_i}{\psi},\frac{e_i}{\psi},\frac{e_j}{\psi})=\frac{1-(\dot{\psi})^2}{(\psi)^2}$$
\begin{Def} $a(s)=\psi^2(K_N-K_T)$, $r_{max}\triangleq |\psi(s)|$
where s is the smallest bump, and $r_{min}\triangleq |\psi(s)|$
where s is the smallest neck.
\end{Def}
Then we could quote the theorem of S.Angenent and D.Knopf:
\begin{thm}
\label{SD} Assume $g_0=ds^2+\psi^2(s)dS^n,\
s\in[-\frac{\pi}{2},\frac{\pi}{2}]$ is a smooth rotationally
symmetric Riemannian metric on $S^{n+1}$. Assume g has one neck, at
least one bump and satisfies:
\begin{enumerate}
\item $K_T>0$, or equivalently $|\dot{\psi}(s)|< 1$.
\item $R>0$
\item $|a|\leq \mu$
\item $\frac{r^2_{max}}{r^2_{min}}\geq \frac{2\mu+2n}{n-1}$
\end{enumerate}
Then $\exists \ \sigma_1(g),\sigma_2(g)$ such that if we evolve $g$
by Ricci-Flow:
$$\partial_tg=-2Ric$$
before time $\frac{r^2_{min(0)}}{n-1}$ we will have:
$$r_{min}(t)\rightarrow 0$$ In particular, $K_T$ at the neck points (\ the point might change along the
flow) will tends to $+\infty$ before  time
$\frac{r^2_{min(0)}}{n-1}$.\\
Furthermore, at the sigular time the volume of the whole manifold
satisfy
$$\sigma_1(g)\leq Vol(g_t)\leq \sigma_2(g)$$
\end{thm}
Although it is trivial, I would like use the following things as an
example of what happens: \begin{prop} Theorem \ref{T2} will not hold
if assumptions 2,3
   in $\digamma(n,p,K,\tau,\eta)$ are simultaneously weaken.  \end{prop}
\begin{proof}Actually S.Angenent and D.Knopf constructed an example
 which satisfy the assumptions in
Theorem \ref{SD}. According to the examples, we could find a
4-dimensional example:\\
 $\exists\ c$ such that: $\psi$ is even function and when $s\leq
c$ we have $g_0=ds^2+(G^2+s^2)dS^3$, for  G sufficiently small.  The
global picture looks like a dumbell.  We compute $|K_N|_{2},
|K_T|_{2}$ to see:
\begin{clm}when G tend to 0
\begin{enumerate}

\item $lim_{G\rightarrow 0}|Rm|_{2[p||s(p)|\leq
c]}$ exists but is far bigger than the energy gap in assumption
$\digamma(n,p,K,\tau,\eta)$.
\item $|Ric|_{p[p||s(p)|\leq c]}$ tends to $\infty$, $p>2$.
\end{enumerate}
  \end{clm}
  First:  $K_N=-\frac{G^2}{(G^2+s^2)^{2}}$ and $K_T=\frac{G^2}{(G^2+s^2)^{2}}$
  Then:
  \begin{eqnarray*}
  & &\int_{(p||s(p)|\leq c)}|K_N|^{2}dVol
  \\&=&2\ \alpha(3)\int^c_0
  [\frac{G^2}{(G^2+s^2)^{2}}]^{2}(G^2+s^2)^{\frac{3}{2}}ds
  \\&\leq&2^{\frac{7}{2}} \alpha(3)\int^c_0\frac{G^{4}}{(G+s)^5}
  \\&=& 2^{\frac{3}{2}}  \alpha(3)[1-\frac{G^4}{(G+c)^4}]
\end{eqnarray*}
also we have:
$$\int_{(p||s(p)|\leq c)}|K_N|^{2}dVol=\int_{(p||s(p)|\leq c)}|K_T|^{2}dVol$$
Then:
$$|Rm|^2_2\leq 4\cdot6\cdot2^{\frac{3}{2}}  \alpha(3)[1-\frac{G^4}{(G+c)^4}]$$
Use $ \alpha(3)=2\pi^2$, we see:

  $$lim_{G\rightarrow 0}|Rm|_{2[p||s(p)|\leq c]}=2^{\frac{11}{4}}\pi\sqrt{3}$$
  It means the Riemannian curvature has a concentration on the
  submanifold $(p|s(p)=0)$.
  Thus it's easy to see $|Ric|_{p[p||s(p)|\leq c]}$ tends to $\infty$,
  $p>2$.\\

  One see that when $G\rightarrow 0$, the metric tensor goes to a
cone pointwisely. Then the volume ratio around the center goes to
$w(4)$.\\
  Using $\alpha(3)=2\pi^2,\ \alpha(4)=\frac{8\pi^2}{3}$, in our case we see
  the energy gap $\varpi(4,\ 1)$ in assumption $\digamma(n,p,K,\tau,\eta)$
  is about   $\frac{7}{8}\pi\cdot 3^{-\frac{1}{2}}2^{-6}$. Thus $lim_{G\rightarrow 0}|Rm|_{2[p||s(p)|\leq
  c]}>100\varpi(4,\ 1)$.
  \\ \\
  Thus  we see the claim holds i.e: according to  Theorem \ref{SD} of Angenent and Knopf, Ricci-Flow starting
  from $g_0=ds^2+(G+s^2)dS^3$ will develop singularity
  around the neck within time $G$. Since the metric is reflection symmetric, and the number of necks will not grow, we see that  after the flow starts,
  the neck does not move, and the size of which will shrink to 0. But c is fixed, thus when G is sufficiently small the solution
  does not satisfy the conclusion of Theorem \ref{T2}.\end{proof}

\begin{rmk}{Something about the Isoperimetric Constant in the above
Example:}\\
One see that when $G\rightarrow 0$, the metric  goes to a cone in
some sense. However, we can not say anything about the geometric
convergence.  I would like to point out that:
\begin{center}For those domain $\Omega_N$ which are bounded by noncontractible $S^3$ in $ [p|s(p)\leq c]$ , the quotients :
 $$\frac{Vol^{\frac{3}{4}}(\Omega_N)}{Vol(\partial \Omega_N)}$$
 are uniformly bounded from above.
\end{center}
It's easy to see that it suffices to bound the above quotients for
those domains whose boundary  are level set of the variable $S$. We
compute $\forall \ b$:
  \begin{eqnarray*}
  & &Vol(p|0\leq s(p)\leq b)
  \\&\simeq&\int^b_0
 [G^2+s^2]^{\frac{3}{2}}ds
  \\&\simeq& [(G+b)^4-G^4]
  \end{eqnarray*}
  Thus consider the domain $\Omega_b\triangleq [p\ |\ s(p)\leq b]$ , we have:

  $$\frac{Vol^{\frac{3}{4}}(\Omega)}{Vol(\partial \Omega)}
  \simeq \frac{[(G+b)^4-G^4]^{\frac{3}{4}}}{ [G^3+b^3]}
$$And $\frac{[(G+b)^4-G^4]^{\frac{3}{4}}}{ [G^3+b^3]}$ is uniformly
bounded from above independent of $G, b$.
 Then we see that for those domain $\Omega_N$ which are bounded by noncontractible $S^3$ in $ [p|s(p)\leq c]$ , the quotients :
 $$\frac{Vol^{\frac{3}{4}}(\Omega_N)}{Vol(\partial \Omega_N)}$$
 are uniformly bounded from above.
 For those domain $\Omega_C$ which are bounded by contractible $S^3$ in $[p|s(p)\leq
 c]$, the quotients $\frac{Vol^{\frac{3}{4}}(\Omega_C)}{Vol(\partial
 \Omega_C)}$ are not known by me yet because there are very hyperbolic parts around the neck no matter how small $G$ is. However,  I think it's also uniformly bounded independent of G and its own size and shape.

\end{rmk}

\section{Set up of Local Ricci flow and Proof of the Main Theorem}
Notation:  $\Omega \Subset M$ means that $\Omega$ is a relatively
compact
 domain of a manifold $M$ (M which might have boundary) ,  $\partial \Omega$ is
 smooth  and $\Omega\backslash\partial\Omega$ is open.
\begin{Def} Suppose $[M,g]$ is   a complete Riemannian Manifold, then
the Local Ricci flow is the deformation of g in the direnction of
$-2\chi^2Ric,\chi\in C^{\infty}_0(M)$ i.e Local Ricci Flow is the
following Geometric PDE:
$$\frac{\partial g}{\partial t}=-2\chi^2Ric, \chi\in
C^{\infty}_0(M)$$ \end{Def} The short time existence of LRF is
proved in \cite{deane}.  One could also refer to \cite{li} .
\begin{Def}Tensor Lipshitz Distance:
\\ Suppose $(M,g_1)$ and $(M,g_2)$  are two compact Riemannian Manifolds with boundary which are diffeomorphic to each other, denote $\dot{M}$ as the interior of $M$. Assume $g_1,g_2$ are nondegenerate on
$\dot{M}$, i.e:
$$\forall \ V_p\neq 0 \in \ T_p\dot{M},\ p\in\dot{M}\  \textrm{we have }\ g_i(V_p,V_p)>0, i=1,2.$$
Then we define the Tensor Lipshitz  Distance between $(M,g_1)$ and
$(M,g_2)$ to be:
$$\displaystyle d_{TL}[(M ,g_1), (M,g_2)]=sup_{p\in \dot{M}}sup_{V_p\in T_p\dot{M}}|log\frac{g_1(V_p,V_p)}{g_2(V_p,V_p)}|$$
\end{Def}
\begin{rmk} The notion of Tensor Lipshitz Distance is more flexible and
local. $(M,g_1)$ and $(M,g_2)$ are not required to be geodesic
convex . But if $(M,g_1)$ and $(M,g_2)$ are all geodesic convex in
their complete manifolds, then Tensor Lipshitz Distance is
equivalent to the ordinary Lipshitz Distance whose definition could
be found in Green-Wu \cite{green wu}.
\end{rmk}

\begin{lem}
\label{TL distance} For a linear space $R^n$, consider the Euclidean
inner product $<,>$, and consider a nondegenerate symmetric matrix
$g_{ij}$, then the following things are equavalent:
\begin{enumerate}
\item $g_{ij}v_iv_j$ is close to  1 for every  $v$ s.t \ $|v|^2=1$
(\ $|\ |^2$ is the Euclidean  square norm).
\item The matrix $g_{ij}$ is close to the identity matrix $I$.
\item $detg_{ij}$ is close to 1, and the determinant of g on any
codimension 1 linear subspace is close to 1.
\end{enumerate}
\end{lem}
\begin{proof} $1\longrightarrow 2$: \\Diagnalize $g$ using the orthogonal
matrix h, namely:

$g=h\cdot Diag(\lambda_1,\lambda_2......\lambda_n)\cdot h^{t}$ in
which $Diag(\lambda_1,\lambda_2......\lambda_n)$ is:
\begin{displaymath}
\left [
\begin{array}{cccccc}
\lambda_1   &               &            &      &        &   \\
            & \lambda_2     &            &      &        &   \\
            &               &  \ddots    &      &         &   \\
            &               &            & \lambda_j     &        &   \\
            &      &        &               &  \ddots      &   \\
            &      &        &      &        &  \lambda_n \\
\end{array}
\right ]
\end{displaymath}
 $(1)$ says that $\lambda_1,\lambda_2......,\lambda_n$ are all close to $1$,
 and because h is orthogonal then $|h|^2=trhh^t=n$ and:
 $$|g-I|=|h(D-I)h^t|\leq n|D-I|$$ then g is close
 to identity matrix.\\
 $2\longrightarrow 3$ :This is obvious.\\
 $3\longrightarrow 1$: Denote $e_j$ is the eigenvector associated to
 $\lambda_j$,then we consider the hyperplane H perpendicular to $e_j$,
 then the determinant of g on H is: $$\lambda_1,\lambda_2...\hat{\lambda j}...\lambda_n$$ in which
 $\hat{\lambda_j}$ means that the product is without $\lambda_j$. By
 the following:
  $$\lambda_j=\frac{\lambda_1,\lambda_2...\lambda_j...\lambda_n}{\lambda_1,\lambda_2...\hat{\lambda_j}...\lambda_n}$$
  Then the assumption that $\lambda_1,\lambda_2...\hat{\lambda_j}...\lambda_n$  and $\lambda_1,\lambda_2...\lambda
 j...\lambda_n$ are   close to 1 tells us that $\lambda_j$ is close to 1.
 j is arbitrary, then we are done.
 \end{proof}

\begin{lem}
\label{uniform convergence inply lipshitz convergence} If we have a
sequence of positive definite matrix function $g^k_{(ij)}$ defined
on $\dot{B_e}(R)$ which satisfy:
\begin{enumerate}
\item $g^k_{(ij)}\rightarrow g^{\infty}_{(ij)}$ uniformly on
$\dot{B_e}(R),\ \forall\ i,j$
\item $\exists \ \lambda_{min}>0$ such that $inf_{k}inf_{p\in \dot{B_e}(R)}g^k_{(ij)}\geq \lambda_{min} I$
\end{enumerate}
Then $g^k_{(ij)}\rightarrow g^{\infty}_{(ij)}$ w.r.t Tensor Lipshitz
distance.
\end{lem}
\begin{proof}: For any vector field $v$ defined on $\dot{B_e}(R)$, denote
$|v|^2$ to the Euclidean square norm. \\
$\forall \ \epsilon$, in the coordinates we have:
$$|g^{\infty}_{(mn)}-g^{j}_{(mn)}|\leq \epsilon$$
Then: \begin{eqnarray} & & \frac{g^{\infty}(v,v)}{g^{j}(v,v)}
\\&=& \frac{g^{\infty}_{(mn)}v^mv^n}{g_{j(mn)}v^mv^n}
\\&\leq&  \frac{g^{j}_{(mn)}v^mv^n+\epsilon|v|^2}{g^{j}_{(mn)}v^mv^n},\ \
\ |v|^2=\sum^{n}_{i=0}v^2_i
\\&\leq& 1+\frac{\epsilon}{\lambda_{min}}
\end{eqnarray}
By the same reason we have:
$$\frac{g^{\infty}(v,v)}{g^{j}(v,v)}\geq 1-\frac{\epsilon}{\lambda_{min}}$$
\end{proof}
\begin{prop}
\label{my iso bound}$\forall \ \epsilon,\ \exists \ \delta$ If
$(M,g_1),(M,g_2)$ satisfy
$$d_{TL}\{(M,g_1),(M,g_2)\}< \delta$$
Then the isoperimetric constant of $(M,g_1)$ and $(M,g_2)$ satisfy:
   $$|\ log\{\frac{C_s(M,g_1)}{C_s(M,g_2)}\}\ |< \epsilon$$\end{prop}
   \begin{proof}We first prove that:  $\forall \ \epsilon$, $\exists \ \delta$ such that if $(M,g_1),(M,g_2)$ satisfy
$$d_{TL}\{\ (M,g_1),(M,g_2)\ \} < \delta$$
   Then  for  any compact
   submanifold $S$ of M we have:
   $$|\ log\{\frac{Vol_{g_1}{S}}{Vol_{g_2}{S}}\}\ |<\epsilon$$
   In particular, $\epsilon$ does not depend on the specific $S$ we
   choose. Next we start the proof.
   Given such a $S$, we could cut $S$ into $k$ pieces:
   $S=\bigcup^k_{i=1}S_i$ such that:
   \begin{itemize}
   \item $dim S_i\bigcap S_j\ \leq dimS -1$
   \item $\forall \ i, \ \exists \ r_i,\ x_i$ such that $B(x_i,r_i)\supset
   S_i$ and $r_i$ strictly smaller than:
   \begin{center} $ min
   \{inj_{g_1|_{M}}(x_i),inj_{g_2|_{M}}(x_i), inj_{g_1|_{S_i}}(x_i), inj_{g_2|_{S_i}}(x_i), dist_{g_1}(x_i,\partial M), dist_{g_2}(x_i,\partial M)
    \}$\end{center}
   \end{itemize}
   Then we have:
   \begin{clm}$\forall \ i$,  if  \ $d_{TL}(g_1,g_2)< \delta$, then:
   $$|log\frac{Vol_{g_1}(S_i)}{Vol_{g_2}(S_i)}|<\frac{m\delta}{2}$$
   in which $m=dim(S_i)$
   \end{clm}
   Proof of the claim: In $B(x_i,r_i),\ \exists$ coordinates
   $(y_1,....y_n)$ such that:
   $$S_i=\{(y_1,....y_n)|y_{m+1}=..=y_n=0\}$$
Thus we use $(y_1,....y_m)$ as coordinates for $S_i$.\\
$\forall\ p\in S_i$, choose $\{u_1,..., u_m\}$ to be the normal
coordinate of $\{g_2, S_i\}$ at $p$. Then
$g_2(\frac{\partial}{\partial u_i},\ \frac{\partial}{\partial u_j})$
is the identity matrix at p. Thus by Lemma \ref{TL distance} and the
assumption:
$$\forall \ V_p\in T_pS_i, \ |log\frac{g_1(V_p,V_p)}{g_2(V_p,V_p)}|\leq \delta$$
We know that:
$$e^{-m\delta}<detg_1{(\frac{\partial}{\partial u_i},\ \frac{\partial}{\partial u_j})}<e^{m\delta}$$
Thus:
$$e^{-\frac{m\delta}{2}}<\frac{\sqrt{detg_1{(\frac{\partial}{\partial u_i},\ \frac{\partial}{\partial u_j})}}}{\sqrt{detg_2{(\frac{\partial}{\partial u_i},\ \frac{\partial}{\partial u_j})}}}<e^{\frac{m\delta}{2}}$$
Then for any other coordinate system, in particular $(y_1,...y_m)$,
we have:
$$e^{-\frac{m\delta}{2}}<\frac{\sqrt{detg_1{(\frac{\partial}{\partial y_i},\ \frac{\partial}{\partial y_j})}}}{\sqrt{detg_2{(\frac{\partial}{\partial y_i},\ \frac{\partial}{\partial y_j})}}}<e^{\frac{m\delta}{2}}$$
then :
$$e^{-\frac{m\delta}{2}}<\frac{\int_{Dom{S_i}}\sqrt{detg_1{(\frac{\partial}{\partial y_i},\ \frac{\partial}{\partial y_j})}}}{\int_{Dom{S_i}}\sqrt{detg_2{(\frac{\partial}{\partial y_i},\ \frac{\partial}{\partial y_j})}}}<e^{\frac{m\delta}{2}}$$which tells us
$$e^{-\frac{m\delta}{2}}<\frac{Vol_{g_1}(S_i)}{Vol_{g_2}(S_i)}<e^{\frac{m\delta}{2}}$$
So we've finished the proof of the claim. \\
Then we see: $$ e^{-\frac{m\delta}{2}}<\frac{\sum_i
Vol_{g_1}(S_i)}{\sum_i Vol_{g_2}(S_i)}<e^{\frac{m\delta}{2}}$$ Which
means:
$$e^{-\frac{m\delta}{2}}<\frac{Vol_{g_1}(S)}{Vol_{g_2}(S)}<e^{\frac{m\delta}{2}}$$
Then $\forall \ \Omega \Subset M$, using  the above proof we see:
$$e^{-(n-1)\delta} <\frac{\frac{Vol_{g_1}(\partial\Omega)}{Vol_{g_1}(\Omega)^{\frac{n-1}{n}}}}{\frac{Vol_{g_2}(\partial\Omega)}{Vol_{g_2}(\Omega)^{\frac{n-1}{n}}}}<e^{(n-1)\delta}$$
Take $\delta=\frac{\epsilon}{n-1}$ which is independent of $\Omega$
then we are done.
   \end{proof}
   Next we state two lemmas without proof.
   \begin{lem}
\label{L1} For all  n and $A ,\ \exists\ c$ such that:\\
If $\{B_p(1), g\}$ is a n-dim Riemannian Unit Ball which satisfy:
$A\cdot Vol{\Omega}^{\frac{n-1}{n}}\leq Vol(\partial\Omega),\
\forall\ \Omega\Subset B(x,1)$.

Then: $$\{VolB_p(\frac{3}{4})\geq c \}$$
\end{lem}

\begin{lem}
\label{L2}
For all  n, $ A,\ \Lambda>0,\ \exists\ i_0$ such that:\\
If $\{B_p(1), g\}$ is a n-dim Unit Ball in a complete Riemannian
Manifold which satisfy:
\begin{enumerate}
\item $|Rm|\leq \Lambda$
\item $A\cdot Vol(\Omega)^{\frac{n-1}{n}}\leq  Vol(\partial\Omega),\  \forall \ \Omega\Subset B(x,1)$
\end{enumerate}
Then: $$inj_g(p)\geq i_0$$
\end{lem}
Now we are ready to go on.
   \begin{prop}
   \label{compactness under bounds on curvature and iso}
For all  $n,\ A\ $ and $\epsilon, \ \exists\ \omega,\ \delta$ such that:\\
If $\{B_p(1), g\}$ is a n-dim Unit Ball in a complete Riemannian
Manifold which satisfy:
\begin{enumerate}
\item $|Rm|\leq \delta$
\item $A\cdot Vol(\Omega)^{\frac{n-1}{n}}\leq  Vol(\partial\Omega),\  \forall \ \Omega\Subset B(x,1)$
\end{enumerate}
Then $B_p(\omega)$ is within distance $\epsilon$ to a Euclidean
Domain w.r.t Tensor Lipshitz Distance.
\end{prop}
\begin{proof}If not, by Lemma \ref{L2}, $\exists$   $\epsilon_0 >0, \ \rho$ and a sequence of $\{B_{p_j}(1), g^j\},\delta_j\rightarrow 0,\omega_j\rightarrow 0$ such that:
\begin{enumerate}
\item $|Rm|\leq \delta_j$
\item $inj_{g^j}(p_j)\geq \rho$
\item $B_{p_{j}}(\omega_j)$ stay further than $\epsilon_0 $ to any
Euclidean Domain w.r.t Tensor Lipshitz Distance.
\end{enumerate}
Using harmonic coordinate theory in page 124 of  \cite{green wu},
the main facts about harmonic balls in page 130 of \cite{green wu},
and the first inequality of the proof of the lemma in page 128 of
\cite{green wu}, we have: $\exists$ constants
$\sigma_1,\sigma_2,\sigma_3$ depending on $n,\rho $ such that:
\begin{enumerate}
\item $\forall \ j, \ \exists \ \phi_j:B_e(\sigma_1)\rightarrow B_{p_j}(1)$
in which $\phi_j$ are harmonic coordinate system and $B_e(\sigma_1)$
is the Euclidean Ball with radius $\sigma_1$. Moreover:
       $$||g^j_{kl}||_{C^{1,\alpha}}\leq const({\alpha,n,\rho})$$
      in which $g^j_{kl}=g^j(\frac{\partial}{\partial h^j_k},\frac{\partial}{\partial h^j_k})$ and  $h^j_1,,,h^j_n$ are  coordinate functions  of $\phi_j$
     . The $C^{1,\alpha}$ norm are also taken w.r.t the
     coordinate.
\item $ B_{p_j}(\sigma_3)\subset \phi_j[B_e(\sigma_1)]\subset
B_{p_j}(\sigma_2)$,
\item $|\phi^{*}_jg^j_{kl}-\delta_{kl}|\leq
\frac{const(n,\rho)\delta^2_j}{1-const(n,\rho)\delta^2_j} $.
\end{enumerate}
Thus $\phi^{*}_jg^j_{kl}$ tends to the Euclidean Metric on
$B_e(\sigma_1)$ uniformly, then from  lemma \ref{uniform convergence
inply lipshitz convergence} we also know that: $\phi^{*}_jg^j_{kl}$
tends to the Euclidean Metric on $B_e(\sigma_1)$ w.r.t Tensor
Lipshitz distance.

Thus it is easy to see that when $j$ is large enough
$B_{p_j}(\sigma_3)$ will be within $\frac{\epsilon_0}{2}$ to a
Euclidean Domain in $B_e(\sigma_1)$ w.r.t Tensor Lipshitz distance.
Contradiction!.
\end{proof}
Then we have:
   \begin{prop}
   \label{strong compactness}
For all  $n, \  A,\ \epsilon, \Lambda, \ \exists \ \hat{\omega} $ such that:\\
If $\{B_p(1), g\}$ is a n-dim Unit Ball in a complete Riemannian
Manifold which satisfy:
\begin{enumerate}
\item $|Rm|\leq \Lambda \ \textrm{in} \ B(x,1)$.
\item $A\cdot Vol(\Omega)^{\frac{n-1}{n}}\leq  Vol(\partial\Omega),\  \forall \ \Omega\Subset
B(x,1)$.
\end{enumerate}
Then $B_p({\hat{\omega}})$ is within distance $\epsilon$ to a
Euclidean Domain w.r.t Tensor Lipshitz Distance.
\end{prop}
\begin{proof}This follows trivially from  Proposition \ref{compactness under bounds on curvature and
iso} and rescaling argument: If not, then $\exists \ \epsilon_0>0$
and $k\rightarrow \infty$ such that  $B_p(\frac{1}{k})$ stay further
than $\epsilon_0$ to any Euclidean domain, then we rescale
$[B_p(\frac{1}{k}), \ g]$ to be $[B_p(1), k^{\frac{1}{2}}g]$ which
will satisfy the assumption of Proposition \ref{compactness under
bounds on curvature and iso} if $k$ is large enough, which is a
contradiction.
\end{proof}
\begin{prop}
\label{noncollapsed} For all $n\geq 3,\ p>\frac{n}{2},
\ A_0,\ K$ and $\epsilon$, there exists  $\upsilon$  such that:\\
Let $[g(t),\ B(x,1)]$  be a  n-dim complete Riemannian unit ball
which satisfy :
\begin{enumerate}
\item $|f|_{\frac{2n}{n-2}}\leq A_0|\nabla f|_{2},\ \forall f\in C_0^{\infty}{[B(x,1)]}$.
\item $|Rm|_{\frac{n}{2}[B(x,1)]}\leq \frac{\beta }{A_0^2},\beta  \leq \ \frac{1}{n^2}$
\item $|Ric|_{p[B(x,1)]}\leq K,\ p>\frac{n}{2}$
\end{enumerate}
Then  $B(x,\upsilon)$ is within distance $\epsilon$ to a Euclidean
Domain w.r.t Tensor Lipshitz Distance.
\end{prop}
\begin{proof}{of Proposition \ref{noncollapsed}:}
 Denote $r(p)$ to be the distance function to x. Then choose:
\begin{displaymath}
\hat{\chi}(p)= \left \{
\begin{array}{ccr}
 1, & \textrm{if} \ 0\leq r\leq\frac{1}{3};\\
  2-3r, & \textrm{if} \  \frac{1}{3}\leq r\leq\frac{2}{3};\\
 0, & \textrm{if} \  \frac{2}{3}\leq r\leq 1
\end{array} \right.
\end{displaymath} We see $\hat{\chi}$ is a lipshitz function with lipshitz
constant
$Lip(\hat{\chi})\leq 3$.\\
Then we could perturb $\hat{\chi}$ to be a function $\chi\in
C^{\infty}_{0}[B(x,1)]$ with $|\nabla\chi|_{\infty}\leq 4$ and:
$$\chi\equiv 1\ in\ B(x,\frac{1}{4}).$$
Therefore we run the LRF:
$$\frac{\partial g}{\partial
t}=-2(\chi^2Ric)$$ According to Theorem \ref{T1}  we have:
$$|\chi^2 Ric|_{\infty}\leq
\bar{C}(n,p){A_0}^{\frac{n}{p}}t^{-\frac{n}{2p}}[t+1]^{\frac{n}{p}+\frac{2}{p}}K$$
within a fixed time. Then $\exists \ t_0$ such that
$$2\int^{t_0}_{0}|\chi^2 Ric|_{\infty}\leq \frac{\epsilon}{2}.$$
Then $\forall\ v\in T{\dot{B}(x,1)}$ we have:
$$e^{-\frac{\epsilon}{2}}g(v,v)\leq g_{t_0}(v,v)\leq e^{\frac{\epsilon}{2}}g(v,v)$$
Thus we have:
  $$d_{TL}\{[B(x,1),g_0],[B(x,1),g_{t_0}]\}\leq \frac{\epsilon}{2}$$
  (Here $B(x,1)$ is the unit ball w.r.t the INITIAL METRIC).\\
  By making $\epsilon$  small enough we could have :
$$ B(x,\frac{1}{16})\subset B_{t_0}(x,\frac{1}{8})\subset B(x,\frac{1}{4})$$
Therefore we have:
$$\chi\equiv 1\ on \ B_{t_0}(x,\frac{1}{8})$$
Thus from  Theorem \ref{T1}  we see:
\begin{itemize}
\item $| Rm|_{\infty(B_{t_0}(x,\frac{1}{8}))}\leq
Const(n)\ t_0^{-1}[t_0+1]^{\frac{4}{n}+2}$
\item The
isoperimetric constant in $[B_{t_0}(x,\frac{1}{8}),\ g_{t_0}]$ is
less than $2A_0$ if $\epsilon$ is small enough.
\end{itemize}
Then according to Proposition \ref{noncollapsed},  $\exists \
\hat{\upsilon}$ such that : $B_{t_0}(x,\frac{\hat{\upsilon}}{8})$
within distance $\frac{\epsilon}{2}$ to a Euclidean Domain w.r.t
Tensor Lipshitz Distance.  We denote the underlying Euclidean Domain
to be
$\Omega_E$.\\
Then we could see from $B(x,\frac{\hat{\upsilon}}{16})\subseteq
B_{t_0}(x,\frac{\hat{\upsilon}}{8})$ and
$d_{TL}\{[B(x,1),g_0],[B(x,1),g_{t_0}]\}< \frac{\epsilon}{2}$ (Here
$B(x,1)$ is the unit ball w.r.t the INITIAL METRIC) that $\exists$ a
Euclidean domain $\hat{\Omega}_E\subset \Omega_E$ s.t:
$$d_{TL}\{B(x,\frac{\hat{\upsilon}}{16}), \ \hat{\Omega}_E\}<
\frac{\epsilon}{2}+\frac{\epsilon}{2}=\epsilon$$

\end{proof}
Then we are ready to prove the following:
\begin{prop}{Strong Noncollapsed:}\\
\label{strong noncollapsed}  For all $n\geq 3,\ p> \frac{n}{2},\
\eta,\ K,\ \tau, \ \epsilon$,  there exists  $\sigma$ such that:\\
If $[g,B(x,1)]$ is a unit Ball in an complete n-dim smooth
Riemannian manifold and satisfies $\digamma(n,p,K,\tau,\eta)$
assumption, then :

      $$(1-\epsilon)C_{E}\cdot Vol(\Omega)^{\frac{n-1}{n}}\leq  Vol(\partial\Omega),\forall\ \Omega\Subset
B(x,\sigma)$$ In which $C_{E}$  is the isoperimetric constant of the
Euclidean space.
\end{prop}

\begin{proof} This is
directly implied by Theorem  \ref{Deane's Iso bounds} , Proposition
\ref{noncollapsed} and \ref{my iso bound}. What is  interesting is
the computation of the constant $\varpi(n,\ \eta)$. Notice that
 $L^2$ Sobolev constant in Theorem \ref{T1} is  $A_0=\frac{2n-2}{n-2}C^{-1}_s,\ C_s$ is the
 underlying isoperimetric constant which can be found in Theorem
 \ref{Deane's Iso bounds}. Furthermore  the $C_s$ in Theorem
 \ref{Deane's Iso bounds} is close to the isoperimetric constant of
 the upper hemisphere.
\end{proof}

Now we are ready to prove Theorem \ref{T2}.

 \begin{proof}{of Theorem \ref{T2}}:\\ It suffices to prove the case: $r=1$.
 By proposition \ref{strong noncollapsed}  we see
 that $\exists \ \delta$ such that the assumption of almost Euclidean isoperimetrc constant in
 Perelman's theorem holds in a smaller ball $B{(x_0,\delta)}$. Then
 using Perelman's Theorem \ref{T3}, we see that the conclusion of Theorem \ref{T3}  holds in a even
 smaller ball: $B_t(x_0,\hat{\epsilon}),\ \hat{\epsilon}=\epsilon
 \delta$ in which $\epsilon$ is exact the one in   Theorem \ref{T3}.

\end{proof}

\section{Extention of Local Ricci Flow under Pointwise Bound of Curvature}
This section is devoted to the proof of Theorem \ref{smooth
extension of LRF}. We just provide the details  of  the  smooth
extension part of Theorem \ref{T1}  which is one of the main
theorems in
     \cite{deane}.

\begin{thm}{Long Time Existence of LRF:}\\
\label{smooth extension of LRF}Suppose $\{g(t), M,\ \chi\in
C_0^{\infty}(M), \ t\in [0,T)\}$ is a n-dim solution to Local Ricci
Flow and $T<\ \infty$. If $|\chi^2 Rm|_{\infty}\leq C$, $C$ is
independent of $t$, then the solution could be extended smoothly
beyond T.
\end{thm}
\begin{rmk} The reason why Theorem \ref{smooth extension of LRF} holds is:\\
Fix $\ m,\ p,\ \exists\ \dot{p},\ \ddot{p},\ \breve{p}$ such that:
\begin{enumerate}
\item $|\nabla^m(\chi^2 Rm)|_p$ depend on  $|\nabla^{m-1}(
Rm)|_{\dot{p}}$ as highest order term.
\item   $|\nabla^{m-1}(
Rm)|_{\dot{p}}$ depend on $|\nabla^{m}\chi|_{\ddot{p}}$ as highest
order term.
\item  $|\nabla^{m}\chi|_{\ddot{p}}$ depend on $|\nabla^{m-1}(\chi^2
Rm)|_{\breve{p}}$ as highest order term.
\end{enumerate}
Then we could use induction to control  $|\nabla^m(\chi^2 Rm)|_p,\
\forall \ p$.
\end{rmk}

We first notice the following extension proposition.

\begin{prop}
\label{weak extension} If $g(t),\  t\in[0,T)$ is a solution to Local
Ricci Flow on $\Omega$:
$$\frac{\partial g}{\partial t}=-2(\chi^2Ric), \chi\in
C^{\infty}_0(\Omega)$$ Then if \ $\forall \ m\geq 0$,
$|\nabla^m(\chi^2 Rm)|_{\infty}$ is uniformly bounded on $[0,T)$,
then the Local Ricci Flow could be extended smoothly through T.
\end{prop}\begin{proof} We use the argument of Hamilton in \cite{HS}.
What is different is that here $\nabla^m(\chi^2Rm)$ plays the
crucial role.
\end{proof}

Next we compute the evolution for $Rm$:

\begin{eqnarray*}
& & \frac{\partial R_{ijkl}}{\partial t}
\\&=& \chi^2[\Delta R_{ijkl}
\\&+&(R_{ijpq}R_{pqkl}+2R_{ipkq}R_{jplq} -2R_{iplq}R_{jpkq})
\\&-&(R_{ip}R_{pjkl}+R_{jp}R_{ipkl}+R_{kp}R_{ijpl}+R_{lp}R_{ijkp})]
\\&-&[(\nabla_{i}\nabla_{k}\chi^2)R_{jl}+(\nabla_{j}\nabla_{l}\chi^2)R_{ik}-(\nabla_{i}\nabla_{l}\chi^2)R_{jk}-(\nabla_{j}\nabla_{k}\chi^2)R_{il}]
\\&-&[(\nabla_{i}\chi^2)\nabla_{k}R_{jl}+(\nabla_{k}\chi^2)\nabla_{i}R_{jl}+(\nabla_{j}\chi^2)\nabla_{l}R_{ik}+(\nabla_{l}\chi^2)\nabla_{j}R_{ik}
\\&-&(\nabla_{j}\chi^2)\nabla_{k}R_{il}-(\nabla_{k}\chi^2)\nabla_{j}R_{il}-(\nabla_{l}\chi^2)\nabla_{i}R_{jk}-(\nabla_{i}\chi^2)\nabla_{l}R_{jk}]
\end{eqnarray*} First, let $T$ and $S$ be two tensors, I would like to use the
"$T\ast S$" to denote all possible algebraic tensor product between
them, just for convenience. And here "$\ast$" is not necessarily
unimodular, i.e $|T\ast S|\leq |T|\cdot|S|$ does not necessarily
holds.
 Second, please notice that unless otherwise
stated , everything depends on time. \\

\begin{lem}\label{curvature}
\begin{eqnarray*}
\displaystyle & &\frac{\partial \nabla^m Rm}{\partial t}
\\&=&\chi^2\Delta\nabla^m
Rm+\sum^{m}_{i+j+k=m+2,\ k \leq m+1}\nabla^i\chi\ast\nabla^j\chi\ast
\nabla^{k}Rm
\\&+&\sum^{m}_{i+j+k=m}\nabla^i{(\chi^2)}\ast\nabla^jRm\ast
\nabla^{k}Rm
\end{eqnarray*}
\end{lem}

\begin{proof} We use induction argument here. Suppose for $m-1$ the
conclusion holds, then we have:
\begin{eqnarray*}
& &\frac{\partial \nabla\nabla^{m-1} Rm}{\partial t}
\\&=&\nabla^{m-1}Rm\ast\nabla(\chi^2Ric)+\nabla[\chi^2\Delta\nabla^{m-1}
Rm] \\&+&\nabla[\sum_{i+j+k=m+1, k\leq
m}\nabla^i\chi\ast\nabla^j\chi\ast \nabla^{k}Rm]
\\&+&\nabla[\sum_{i+j+k=m-1}\nabla^i{(\chi^2)}\ast\nabla^j Rm\ast
\nabla^{k}Rm]
\end{eqnarray*}
Notice that
\begin{eqnarray*}
\\& & \nabla[\chi^2\Delta\nabla^{m-1}Rm]
\\&=& \chi^2
\Delta\nabla^{m}Rm+\chi^2\nabla^{m-1}Rm\ast\nabla Rm
\\&+&\chi^2\nabla^{m}Rm\ast Rm+(\nabla\chi^2)\ast\Delta\nabla^{m-1}Rm
\end{eqnarray*}
If we consider
$$\nabla^{m-1}Rm\ast\nabla(\chi^2Ric
),\ \chi^2\nabla^{m-1}Rm\ast\nabla Rm,\ \chi^2\nabla^{m}Rm\ast Rm$$
as
 the terms in
$\nabla[\sum_{i+j+k=m-1}\nabla^i{(\chi^2)}\ast\nabla^j Rm\ast
\nabla^{k}Rm]$ and  $$(\nabla\chi^2)\ast\Delta\nabla^{m-1}( Rm)$$ as
one of the terms in $$\nabla[\sum_{i+j+k=m+1,\ k\leq
m}\nabla^i\chi\ast\nabla^j\chi\ast \nabla^{k}Rm]$$ then we are done.
\end{proof}

In the same way we could see:
\begin{lem}\label{cutoff function times curvature}
\begin{eqnarray*}
& & \frac{\partial \nabla^m (\chi^2Rm)}{\partial t}
\\&=&\chi^2\Delta\nabla^m
(\chi^2Rm)+\sum^{m}_{i+j+k=m+2, \ k\leq \
m+1}\nabla^i\chi\ast\nabla^j\chi\ast \nabla^{k}(\chi^2Rm)
\\&+&\sum^{m}_{i+j+k=m,\ k\leq \ m-1}\nabla^i{(\chi^2)}\ast\nabla^j(\chi^2Rm)\ast
\nabla^{k}Rm
\\&+&\chi^2Rm\ast\nabla^m(\chi^2Rm)
\end{eqnarray*}
\end{lem}
\begin{lem}\label{cutoff function}
\begin{eqnarray*}\displaystyle
& &\frac{\partial}{\partial t}|\nabla^m\chi| \leq C(n,k)\sum_{k\leq
m-1}|\nabla^{k}(\chi^2Ric)|\cdot|\nabla^{m-k}\chi|
\end{eqnarray*}
\end{lem} And we should state the following lemmas:
\begin{lem}
\label{ODE2}If $f$ and $g$ are  nonnegative functions of t ,
$t\in[0,T)$, b is a nonnegative constant. Assume $f$ satisfy:
  $$\frac{df}{dt}\leq g+bf
  $$
  Then on $[0,T)$ we have :

   $$f\leq [f(0)+\int^t_0e^{-bs}gds]e^{bt}$$

\end{lem}

\begin{lem}
\label{nontrivial ODE} If $y(t)$ is nonnegative functions of t ,
$t\in[0,T)$. a , b are nonnegative constants. Assume \ $y$ satisfy:
  $$\frac{dy}{dt}\leq a\int^{t}_0yds +by+1$$
  Then  $\exists \ w,\ k$ which depend on $a, \ b$ and $y(0)$ such
  that:

   \begin{eqnarray*}
y\ \leq \ we^{kt}, \ \forall t\in \ [0,T).
\end{eqnarray*}

\end{lem}
\begin{proof}It suffices to choose $w,\ k$ \ large enough to ensure:
$$(we^{kt})'> a\int^{t}_0(we^{kt})ds +b(we^{kt})+1 \ \ \textrm{and}\
w>\ y(0)$$ to get the comparison function.
\end{proof}
\begin{prop}
\label{estimate for cutoff function times curvature}If  $|\chi^2
Rm|_{\infty}$ is uniformly bounded along a given solution of LRF
$g(t), t\in [0,\ T)$,\  then $ |\nabla\chi|_{\infty}$ and the
Sobolev constant $A$ are uniformly bounded on $[0,T)$.
\end{prop}
\begin{proof}One only have to notice that the metrics $g(t), t\in [0,\ T)$ are uniformly
equivalent to each other.
\end{proof}
From now we will do curvature estimate. I will denote all the
quantities which are  uniformly bounded  independent of time on
$[0,T)$ by $C$. The $C$ in different places may have different
meaning, but it should be clear from the context that what does the
$C$ represent. If some quantity is assumed to be uniformly bounded
on $[0,T)$ in a specific lemma, I would also denote it by $C$.
\begin{lem}\label{curvature integral}
$\forall \ p> 100$. If \ $g(t), t\in[0,T)$ is a smooth solution to
$LRF$ on M and satisfy:

$|\chi^2Rm|_{\infty} $ and $|\nabla\chi|_{\infty}$ are uniformly
bounded  on $[0,T)$, \\ then $|Rm|_p$ is uniformly bounded on $[0,\
T)$.

\end{lem}
\begin{proof}Recall that

\begin{eqnarray*}
& &\frac{\partial}{\partial t}\int|Rm|^p
\\&=&p\{-\int\chi^2|\nabla Rm|^{2}|Rm|^{p-2}
+\int|Rm|^{p-2}<\chi\nabla\chi\ast\nabla Rm,Rm>
\\&+&\int| Rm|^{p-2}<\chi^2Rm\ast
Rm,Rm>
\\&+&\int| Rm|^{p-2}<\nabla\chi\ast\nabla\chi\ast
Rm,Rm> +\int| Rm|^{p-2}<\chi\nabla^2\chi\ast Rm,Rm> \}
\end{eqnarray*}
Apply integration by parts to last term yields:
\begin{eqnarray*} &
&\int|Rm|^{p-2}<\chi\nabla^2\chi\ast Rm,Rm>
\\&=&
-(p-2)\int|Rm|^{p-4}<\nabla Rm,Rm><\chi\nabla\chi\ast Rm,Rm>
\\&+&
\int|Rm|^{p-2}<\chi\nabla\chi\ast\nabla Rm,Rm>
\\&+&
\int|Rm|^{p-2}<\nabla\chi\ast\nabla\chi\ast Rm,Rm>
\\&+&
\int|Rm|^{p-2}<\chi\nabla\chi\ast Rm,\nabla Rm>
\\&\leq&C\int|Rm|^{p}+\frac{1}{16}\int\chi^2|\nabla Rm|^{2}|Rm|^{p-2}
\end{eqnarray*}
In which we used:
\begin{eqnarray*} &
&\int|Rm|^{p-2}<\chi\nabla\chi\ast\nabla Rm,Rm>
\\&\leq&C\int|Rm|^{p}+\frac{1}{16}\int\chi^2|\nabla Rm|^{2}|Rm|^{p-2}
\end{eqnarray*}
Then:
$$\frac{\partial}{\partial t}\int|Rm|^p\leq C\int|Rm|^p$$
Thus $|Rm|_p$ is uniformly bounded on $[0,\ T)$.

\end{proof} Then we are ready to bound the integral of the first derivative
of curvature:
\begin{lem}
\label{first derivative of curvature and hessian of cutoff fuction}
$\forall \ p> 100$. If $g(t), t\in[0,T)$ is a smooth solution to
$LRF$ on M assume further on $[0,T)$:
\begin{itemize}
\item $|Rm|_{q}$ is  uniformly  bounded,$ \forall \ q>100$.
\item $|\chi^2Rm|_{\infty} $ and $|\nabla\chi|_{\infty}$ are uniformly
bounded.
\end{itemize}
  Then we have: $|\nabla
(\chi^2Rm)|_p$ and $|\nabla^2\chi|_{p}$ are uniformly bounded on
$[0,T)$.
\end{lem}
\begin{rmk} The reason why we should do this lower order estimate first is $|\nabla^2\chi|_q$ and $|\nabla(\chi^2 Rm)|_q$ are interplaying
along the flow, and both of them are crucial for us to go on with
induction!
\end{rmk}

\begin{proof} First  from  lemma \ref{ODE2} we have:
 $$|\nabla^2\chi|_{p}\leq C[|\nabla^2\chi|_{p}(0)+\int^t_0|\nabla(\chi^2Ric)|_pds]$$
So we could see:
\begin{eqnarray*}& &
\\& & \frac{\partial \int|\nabla
(\chi^2Rm)|^p}{\partial t} \leq p\sum^{9}_{i=0}I_i.
\end{eqnarray*}

\begin{itemize}
\item
$I_0=-\int\chi^2|\nabla^2(\chi^2Rm)|^{2}|\nabla(\chi^2Rm)|^{p-2}$
\item $I_1=\int|\nabla (\chi^2Rm)|^{p-2}<\nabla
(\chi^2Rm),\chi\ast\nabla^3\chi\ast (\chi^2Rm)>$
\item $I_2=\int|\nabla (\chi^2Rm)|^{p-2}<\nabla
(\chi^2Rm),\chi\ast\nabla^2\chi\ast \nabla (\chi^2Rm)>$
\item $I_3=\int|\nabla (\chi^2Rm)|^{p-2}<\nabla
(\chi^2Rm),\nabla\chi\ast\nabla\chi\ast \nabla (\chi^2Rm)>$
\item $I_4=\int|\nabla (\chi^2Rm)|^{p-2}<\nabla (\chi^2Rm),
\nabla\chi\ast\nabla^2\chi\ast (\chi^2Rm)>$
\item $I_5=\int|\nabla (\chi^2Rm)|^{p-2}<\nabla
(\chi^2Rm), \chi\nabla\chi\ast \chi^2Rm\ast Rm>$
\item $I_6=\int|\nabla
(\chi^2Rm)|^{p-2}<\nabla (\chi^2Rm), \chi^2Rm\ast\nabla(\chi^2Rm)>$
\item $I_7=\int|\nabla
(\chi^2Rm)|^{p-2}<\nabla (\chi^2Rm), \chi\nabla\chi\ast
\nabla^2(\chi^2Rm)>$
\end{itemize} It's trivial to see that:
$$  I_5\leq C|\nabla(\chi^2Rm)|^{p-1}_{p}$$
$$I_3,\ I_6 \ \leq C|\nabla(\chi^2Rm)|^{p}_{p}$$
$$  I_4\leq C|\nabla(\chi^2Rm)|^{p-1}_{p}|\nabla^2\chi|_{p}$$ The biggest
trouble come from $I_1,\ I_2$ because they involve second and third
derivatives of $\chi$. First we estimate  $I_2$:
\begin{eqnarray*}
&  I_2& =\ -\int|\nabla (\chi^2Rm)|^{p-2}<\nabla^2
(\chi^2Rm),\chi\nabla\chi\ast \nabla (\chi^2Ric)>
\\&-&\int|\nabla (\chi^2Rm)|^{p-2}<\nabla
(\chi^2Rm),\chi\nabla\chi\ast \nabla^2 (\chi^2Ric)>
\\&-&\int|\nabla (\chi^2Rm)|^{p-2}<\nabla
(\chi^2Rm),\nabla\chi\ast\nabla\chi\ast \nabla (\chi^2Ric)>
\\&-(p-2)&\int|\nabla
(\chi^2Rm)|^{p-4}<\nabla^2 (\chi^2Rm),\nabla (\chi^2 Rm)><\nabla
(\chi^2Rm),\chi\nabla\chi\ast \nabla(\chi^2Ric)>
\end{eqnarray*}
Then we estimate:
\begin{eqnarray*}
& &|-\int|\nabla (\chi^2Rm)|^{p-2}<\nabla^2
(\chi^2Rm),\chi\nabla\chi\ast \nabla(\chi^2Ric)>|
\\&\leq&\frac{1}{16}\int\chi^2|\nabla^2 (\chi^2Rm)|^2|\nabla
(\chi^2Rm)|^{p-2} + C\int|\nabla(\chi^2Rm)|^{p}
\end{eqnarray*}
Estimate the other terms in the same way  we have:
\begin{eqnarray*}
& &I_2\  \leq\ \frac{1}{8}\int\chi^2|\nabla^2 (\chi^2Rm)|^2|\nabla
(\chi^2Rm)|^{p-2} +\ C\int|\nabla{(\chi^2Rm)}|^{p}
\end{eqnarray*}
Then we estimate $I_1$:
\begin{eqnarray*}
\\&I_1=&-\int|\nabla
(\chi^2Rm)|^{p-2}<\nabla^2 (\chi^2Rm),\chi\nabla^2\chi\ast
(\chi^2Ric)>
\\&-(p-2)&\int|\nabla
(\chi^2Rm)|^{p-4}<\nabla^2 (\chi^2Rm),\nabla (\chi^2 Rm)><\nabla
(\chi^2Rm),\chi\nabla^2\chi\ast (\chi^2Ric)>
\\&-&\int|\nabla
(\chi^2Rm)|^{p-2}<\nabla (\chi^2Rm),\nabla\chi\ast\nabla^2\chi\ast
(\chi^2Ric)>
\\&-&\int|\nabla
(\chi^2Rm)|^{p-2}<\nabla (\chi^2Rm), \chi\nabla^2\chi\ast
\nabla(\chi^2Ric)>
\end{eqnarray*} Notice that:
\begin{eqnarray*}
& &|-\int|\nabla (\chi^2Rm)|^{p-2}<\nabla^2
(\chi^2Rm),\chi\nabla^2\chi\ast (\chi^2Ric)>|
\\&\leq&\frac{1}{16}\int\chi^2|\nabla^2 (\chi^2Rm)|^2|\nabla
(\chi^2Rm)|^{p-2} + C\int|\nabla^2\chi|^2| (\chi^2Ric)|^2|\nabla
(\chi^2Rm)|^{p-2}
\end{eqnarray*}
and also \begin{eqnarray*} & &(p-2)\int|\nabla
(\chi^2Rm)|^{p-4}<\nabla^2 (\chi^2Rm),\nabla (\chi^2 Rm)><\nabla
(\chi^2Rm),\chi\nabla^2\chi\ast (\chi^2Ric)>
\\&\leq&\frac{1}{16}\int\chi^2|\nabla^2 (\chi^2Rm)|^2|\nabla
(\chi^2Rm)|^{p-2} +\ C|\nabla^2\chi|_{p}^2|\nabla
(\chi^2Rm)|_{p}^{p-2}
\end{eqnarray*}
The other two terms in $I_1$ are the same as those in $I_2,I_4$. Thus we get:
\begin{eqnarray*}
& &I_1
\\&\leq&\frac{1}{16}\int\chi^2|\nabla^2 (\chi^2Rm)|^2|\nabla
(\chi^2Rm)|^{p-2} \\&+& C|\nabla^2\chi|_{p}^2|\nabla
(\chi^2Rm)|_{p}^{p-2}
\\&+& C|\nabla^2\chi|_{p}|\nabla
(\chi^2Rm)|_{p}^{p-1}+C|\nabla (\chi^2Rm)|_{p}^{p}
\end{eqnarray*}

The estimate of $I_7$ is the same as the technipue of one of the
terms in $I_2$. \\
Then the conclusion is:
\begin{eqnarray*}
& & \frac{\partial \int|\nabla (\chi^2Rm)|^p}{\partial t}
\\&=&\{C|\nabla (\chi^2Rm)|^{p-2}_{p}[|\nabla^2\chi|_{p}(0)+\int^t_0|\nabla(\chi^2Ric)|_pds]^2
\\&+&C|\nabla (\chi^2Rm)|^{p-1}_{p}[|\nabla^2\chi|_{p}(0)+\int^t_0|\nabla(\chi^2Ric)|_pds]
\\&+&C|\nabla(\chi^2Rm)|^{p-1}_{p}
+C\int|\nabla(\chi^2Rm)|^{p} \ \
\end{eqnarray*}
Then we use Schwarz inequality and Holder Inequality  to see:
\begin{eqnarray*}
& & \frac{d}{\partial t} |\nabla (\chi^2Rm)|^{2}_p \leq
C\int_0^t|\nabla (\chi^2Rm)|^{2}_pds +C|\nabla (\chi^2Rm)|^{2}_p+1
\end{eqnarray*}
Then from lemma \ref{nontrivial ODE} we see that $|\nabla
(\chi^2Rm)|^{2}_p$ is uniformly bounded on $[0,T)$, thus
$|\nabla^2\chi|_{p}$ is also uniformly bounded on $[0,T)$.

\end{proof}

\begin{prop}
\label{bounds on  derivative of Rm} If $g(t), t\in[0,T)$ is a smooth
solution to $LRF$ on M. Assume on $[0,T)$ we have:
\begin{enumerate}
\item $\forall \ 0\leq i\leq m+1,|\nabla^{i} \chi|_{q}$ are uniformly
bounded, $\forall \ q>100$
\item  $\forall \ q>100,\ |Rm|_q$ is uniformly bounded.
\item $|\chi^2Rm|_{\infty} $ and $|\nabla\chi|_{\infty}$ are uniformly
bounded.
\end{enumerate}

  Then  $\forall \ p>100, \ |\nabla^m
Rm|_p$ is uniformly bounded on $[0,T)$.
\end{prop}
\begin{proof}
We argue by induction. When $m=0$ the conclusion  is right. For all
$m\geq 1$, assume that  $\forall \ 0\leq k\leq m-1, \ p>100$:
$$\int|\nabla^k
Rm|^{p}$$ is uniformly bounded. Then we compute:
\begin{eqnarray*}\displaystyle
& & \frac{\partial \int|\nabla^m Rm|^p}{\partial t}
\\&\leq&C\{\sum_{i+j+k=m+2,\ i,\ j,\ k\leq m+1}\int|\nabla^i\chi||\nabla^j\chi||\nabla^k(Rm)||\nabla^m Rm|^{p-1}
\\&+&\sum_{i+j+k=m}\int|\nabla^i(\chi^2)||\nabla^j Rm||\nabla^k(Rm)||\nabla^m Rm|^{p-1}
\\&+&\int|\nabla^m Rm|^{p-2}<\nabla^m
Rm,\chi\nabla^{m+2}\chi \ast Rm>\}
\\&-&p\int\chi^2|\nabla^{m+1}Rm|^{2}|\nabla^{m}Rm|^{p-2}
\end{eqnarray*}
Apply integration by parts to the  term $\int|\nabla^m
Rm|^{p-2}<\nabla^m Rm,\chi\nabla^{m+2}\chi \ast Rm>$ and use the
same technique as we did in
 Lemma \ref{first derivative of curvature and hessian of cutoff
 fuction} (integration
by parts, Schwarz Inequality and holder inequality) as well as the
induction hypothesis to see:

\begin{eqnarray*}& & \frac{\partial \int|\nabla^m
Rm|^p}{\partial
t}+\frac{p}{2}\int\chi^2|\nabla^{m+1}Rm|^{2}|\nabla^m Rm|^{p-2}
\\&\leq& C\int|\nabla^m
Rm|^p +C|\nabla^m Rm|_{p}^{p-1}+C
\end{eqnarray*}
 then $ \int|\nabla^m Rm|^{p}$ is uniformly bounded on $
[0,T)$.
\end{proof}
\begin{prop}
\label{Extension proposition }. If $g(t), t\in[0,T)$ is a smooth
solution to $LRF$ on M and assume on $[0,T)$ the following hold:
\begin{itemize}
\item $\forall \ q>100,\ |Rm|_{q},\ |\nabla^2\chi|_q$ and $|\nabla(\chi^2Rm)|_q$ are  uniformly
bounded.
\item $ \forall \ q>100,\ |\chi^2Rm|_{\infty} $ and $|\nabla\chi|_{\infty}$ are uniformly
bounded.
\item The Sobolev constant A is uniformly bounded.
\end{itemize}
  Then $\forall \ m\geq 0,\ |\nabla^m(\chi^2Rm)|_{\infty}$ is uniformly bounded on
  $[0,T)$.\end{prop}

\begin{proof} Since   Sobolev constant A is assumed to be uniformly
  bounded on $[0,T)$, then it suffices to prove:
   $ \forall \ p>100,\ m\geq 2,\ |\nabla^m(\chi^2Rm)|_{p}
  $ is uniformly bounded on
  $[0,T)$.  We will prove this by induction. \\Assume that : $|\nabla^l
(\chi^2Rm)|_p$ is uniformly bounded on $[0,T),\forall \ p>100,\
l\leq m-1$, we should show  $|\nabla^m (\chi^2Rm)|_p$ is uniformly
bounded on
$[0,T),\forall \ p>100$.\\
 The induction hypothesis, Lemma \ref{cutoff function}, Proposition \ref{bounds on  derivative of Rm} and bound of Sobolev Constant  tells us that:
\begin{itemize}
\item $|\nabla^l (\chi^2Rm)|_{\infty}$
is uniformly bounded on \ $[0,T),\  l\leq m-2$.
\item $|\nabla^k \chi|_p$ is uniformly  bounded on  $[0,T)$,
$\forall\ p>100,\  k\leq m.$
\item $|\nabla^j \chi|_{\infty}$ is uniformly  bounded on  $[0,T)$,
$j\leq m-1.$
\item $|\nabla^i
Rm|_p$ is uniformly  bounded on $ [0,T),\ p>100,\ i\leq m-1.$
\end{itemize}

We compute $\forall \ p$:
\begin{eqnarray*}\displaystyle
& & \frac{\partial}{\partial t} \int|\nabla^m (\chi^2Rm)|^p
\\&\leq&
-\int\chi^2|\nabla^{m+1}(\chi^2Rm)|^{2}|\nabla^{m}(\chi^2Rm)|^{p-2}
\\&+&\sum_{i+j+k=m+2,\ k\leq m+1}\int|\nabla^m
(\chi^2Rm)|^{p-2}<\nabla^m(\chi^2Rm),\nabla^i\chi\ast\nabla^j\chi\ast
\nabla^k(\chi^2Rm)>
\\&+&\sum_{i+j+k=m,\ k\leq m-1}\int|\nabla^m
(\chi^2Rm)|^{p-2}<\nabla^m(\chi^2Rm),\nabla^i(\chi^2)\ast
\nabla^j(\chi^2Rm)\ast\nabla^kRm>
\\&+&\int|\nabla^m
(\chi^2Rm)|^{p-2}<\nabla^m(\chi^2Rm),\chi^2Rm\ast\nabla^m(\chi^2Rm)>
\\&\leq&1+\ C\int|\nabla^m (\chi^2Rm)|^p +\Gamma_1+\Gamma_2+\Gamma_3+\Gamma_4
\end{eqnarray*}
In which:
\begin{eqnarray*}
& & \Gamma_1 =\int|\nabla^m
(\chi^2Rm)|^{p-2}<\nabla^m(\chi^2Rm),\chi\nabla^{m+2}\chi\ast
\chi^2Rm> \\& &\Gamma_2 =\int|\nabla^m
(\chi^2Rm)|^{p-2}<\nabla^m(\chi^2Rm),\nabla\chi\ast\nabla^{m+1}\chi\ast
\chi^2Rm> \\& &\Gamma_3 =\int|\nabla^m
(\chi^2Rm)|^{p-2}<\nabla^m(\chi^2Rm),\chi\ast\nabla^{m+1}\chi\ast
\nabla(\chi^2Rm)>\\& & \Gamma_4 =\int|\nabla^m
(\chi^2Rm)|^{p-2}<\nabla^m(\chi^2Rm),\chi\nabla\chi\ast
\nabla^{m+1}(\chi^2Rm)>
\end{eqnarray*}
Notice $\Gamma_4\leq  C\int|\nabla^m (\chi^2Rm)|^p+
\frac{1}{16}\int\chi^2|\nabla^{m+1}(\chi^2Rm)|^{2}|\nabla^{m}(\chi^2Rm)|^{p-2}$.\\
When we have applied  integration by parts to $\Gamma_1, \Gamma_2,
\Gamma_3$ we can control the terms in the way we control $\Gamma_4$
except the following term:
$$\int|\nabla^m
(\chi^2Rm)|^{p-2}<\nabla^m(\chi^2Rm),\nabla^2\chi\ast\nabla^{m}\chi\ast
\chi^2Rm> $$ But we have assumed $ |\nabla^2\chi|_q$ is  uniformly
bounded on $[0,T),\ \forall \ q>100$. Then we are done! Actually
this is why we should bound $|\nabla^2\chi|_q$ first : If $m=2$,
then the above term contain $|\nabla^2\chi|^2$ then we should deal
with the interplay between  $|\nabla^2\chi|_q$ and $|\nabla(\chi^2
Rm)|_q$.
\end{proof}
\begin{proof}{of Theorem \ref{smooth extension of LRF}:}
It is directly implied by Lemma \ref{curvature integral}, \ref{first
derivative of curvature and hessian of cutoff fuction},\ \ref{bounds
on derivative of Rm}, \ Proposition \ref{estimate for cutoff
function times curvature}, \ref{weak extension} and \ref{Extension
proposition }.\end{proof}
\section{Proof of the Smooth Extension Part of Theorem \ref{T1}}

We only compute the gap constant in Assumption 2 of Theorem \ref{T2}
and apply Theorem \ref{smooth extension of LRF}, the rest are the
same as the computation from page 98 to page 102 of \cite{deane}. \\
Notice that:
\begin{itemize}
\item $|<Rm^2 +Rm\times Rm,Rm>|\leq \ 5|Rm|^3$
\item $|R|\leq \sqrt{\frac{n(n-1)}{2}}|Rm|$\end{itemize}
In which \ $Rm^2 +Rm\times Rm=\ R_{ijpq}R_{pqkl}+2R_{ipkq}R_{jplq}
-2R_{iplq}R_{jpkq}$. Then we compute:
\begin{eqnarray*}
\\& &\frac{\partial}{\partial t}\int|Rm|^{\frac{n}{2}}
\\&\leq&\frac{n}{2}\{6.5\int\chi^{2}|Rm|^{\frac{n}{2}+1}-\frac{16}{n^2}
\int
|\nabla(\chi|Rm|^{\frac{n}{4}})|^2\}+C(n)|\nabla\chi|^2_{\infty}\int
|Rm|^{\frac{n}{2}}
\\&+& C(n)\int|\nabla\chi|_{\infty}|\nabla(\chi|Rm|^{\frac{n}{4}})|\
|Rm|^{\frac{n}{4}}
\\&\leq&\frac{n}{2}\{6.5\int(\chi|Rm|^{\frac{n}{4}})^{2}|Rm|-\frac{15.5}{n^2}
\int
|\nabla(\chi|Rm|^{\frac{n}{4}})|^2\}+C(n)|\nabla\chi|^2_{\infty}\int
|Rm|^{\frac{n}{2}}
\\&\leq&\frac{n}{2}\{6.5[|Rm|_{\frac{n}{2}}\int(\chi|Rm|^{\frac{n}{4}})^{\frac{2n}{n-2}}]^{\frac{n-2}{n}}-\frac{15}{A^2n^2}\int(\chi|Rm|^{\frac{n}{4}})^{\frac{2n}{n-2}}]^{\frac{n-2}{n}}\}+C(n)|\nabla\chi|^2_{\infty}\int |Rm|^{\frac{n}{2}}
\\&-& \frac{1}{4n}\int
|\nabla(\chi|Rm|^{\frac{n}{4}})|^2
\end{eqnarray*}

   From  the same analysis  from page 98 to page 102 of \cite{deane}, we see if we assume at $t=0$ that:
   $$|Rm|_{\frac{n}{2}}|_0\leq
  \frac{2}{n^2A_0^2}$$
  and shrink the constant $C(n,p)$ in Theorem
 \ref{T1} to be sufficiently small, we have:
$$\frac{\partial}{\partial t}\int |Rm|^{\frac{n}{2}}+\frac{1}{4n}\int
|\nabla(\chi|Rm|^{\frac{n}{4}})|^2\leq
C(n)|\nabla\chi|^2_{\infty}\int |Rm|^{\frac{n}{2}}$$ and
$$n^2A^2|Rm|_{\frac{n}{2}}\leq \frac{15}{6.5}, \ \textrm{in} \ [0,T)$$
provided $$T\leq
C(n,p)min[\frac{1}{|\nabla\chi|^2_{\infty}|_0},A_0^{-\frac{2n}{2p-n}}K^{-\frac{2p}{2p-n}}]$$
 Furthermore, in order to fulfill   the  step of interpolation in the  Moser Iteration    from page 98 to page 102 of \cite{deane}, we should compute:
 \begin{eqnarray*}
\\& &\frac{\partial}{\partial t}\int\chi^2|Rm|^{1+\frac{n}{2}}
\\&\leq&(1+\frac{n}{2})\{6.5|Rm|_{\frac{n}{2}}(\int(\chi^2|Rm|^{\frac{n}{4}+\frac{1}{2}})^{\frac{2n}{n-2}})^{\frac{n-2}{n}}-\frac{15}{(n+2)^2}\int|\nabla(\chi^2|Rm|^{\frac{n}{4}+\frac{1}{2}}|^{2}\}
\\&+& C(n)|\nabla\chi|^2_{\infty}\int |Rm|^{\frac{n}{2}}
\end{eqnarray*}
if we assume at $t=0$ that:
   $$|Rm|_{\frac{n}{2}}|_0\leq
  \frac{1}{2n^2A_0^2}$$
  and shrink the constant $C(n,p)$ in Theorem
 \ref{T1} to be even smaller, we have:
$$\frac{\partial}{\partial t}\int\chi^2|Rm|^{\frac{n}{2}+1}\leq
C(n)|\nabla\chi|^2_{\infty}\int\chi^2|Rm|^{\frac{n}{2}+1}$$ and
$$(n+2)^2A^2|Rm|_{\frac{n}{2}}\leq \frac{15}{6.5}, \ \textrm{in} \
[0,T)$$Which give us $$\int\chi^2|Rm|^{\frac{n}{2}+1}\leq
C(n)A^{-n}t^{-1}(t|\nabla\chi|^2_{\infty}+1)^2
$$ Thus the interpolation argument in moser iteration works. Apply
parabolic moser iteration from page 98 to page 102 of \cite{deane}
we see that $|\chi^2Rm|_{\infty}$ is uniformly bounded on $[0,T)$.\\
Then the smooth extension part of Theorem \ref{T1} is impied by
Theorem \ref{smooth extension of LRF}.
\section{Appendix}
Here we state several theorems  we used:
 \begin{thm}(Perelman:)\\
 \label{T3}For all n and
$\alpha> 0$
there exists $\epsilon$ and $\delta$ such that:\\
If $[g(t),M]$ is a compact solution to Ricci-Flow: $\frac{
\partial g}{\partial t}=-2Ric$, , $0\leq t\leq (\epsilon r)^2$ and when $t=0$
satisfies:

\begin{enumerate}
\item $(1-\delta)C_EVol(\Omega)^{\frac{n-1}{n}}\leq \cdot Vol(\partial\Omega),\forall\
\Omega\Subset B_0(x_0,r)$,  $C_E$  is the isoperimetric constant of
the Euclidean space.
\item $R\geq -r^2$ on $B_0(x_0,r)$
\end{enumerate}
Then $|Rm|(x)\leq\frac{\alpha}{t}+(\epsilon r)^{-2}$ and
 $VolB(x,t^{\frac{1}{2}})\geq C(n)t^{\frac{n}{2}}$ whenever
$0< t\leq (\epsilon r)^2$ and $d_t(x,x_0)\leq \epsilon r$.

\end{thm}

  \begin{thm}{Deane Yang:}
\label{Deane's Iso bounds} For all $n$,  $p>\frac{n}{2}$, $\eta
,\epsilon$ and $K$, $\exists \ \omega_0$ and $R$ such that $\forall\
\tau\leq \omega_0,\ \exists R$  such that:\ If $[B(x,1), g]$ be an
unit ball in a complete n-dim Riemannian Manifold which satisfy:
\begin{enumerate}
\item  $Vol\ B_0(x_0,\tau )\geq w(n)(\eta\tau)^n$
\item $|Ric_{-}|_{p[B(x,1)]}\leq K,\ p>\frac{n}{2}$
\end{enumerate}
Then we have:
\begin{center}$(1-\epsilon)2^{1-\frac{1}{n}}\alpha(n-1)\alpha^{\frac{1}{n}-1}(n)\eta^{n+1}Vol(\Omega)^{\frac{n-1}{n}}\leq  Vol(\partial\Omega), \forall\ \Omega\Subset
B(x,R)$. \end{center}
\end{thm}
\begin{thm}{Deane Yang:}\\
\label{T1} For all $n\geq 3,\ p>\frac{n}{2},\ \beta  \leq
\frac{1}{2n^2}
$ and positive constants $A,K$  we have :\\
Let $[g(t),\ M]$  be a  n-dim complete Riemannian manifold and
$\chi\in \ C_0^{\infty}{(M)}$. If \ $[g(t),M,\chi]$ satisfy:
\begin{enumerate}
\item $|f|_{\frac{2n}{n-2}}\leq A_0|\nabla f|_{2},\ \forall \ f\in C_0^{\infty}{(M)}$.
\item $|Rm|_{\frac{n}{2}[M]}\leq \frac{\beta }{A_0^2},\ \beta  \leq \ \frac{1}{2n^2}$
\item $|Ric|_{p[M]}\leq K,\ p>\frac{n}{2}$
\end{enumerate}
Then the local Ricci flow :$$\frac{\partial g}{\partial
t}=-2(\chi^2Ric)$$ has a smooth solution up to time:
$$T=C(n,p)min[\frac{1}{|\nabla\chi|^2_{\infty}|_0},A_0^{-\frac{2n}{2p-n}}K^{-\frac{2p}{2p-n}}]$$
And when $t\in(0,T]$ we have:
\begin{enumerate}
\item $|\chi^2 Rm|_{\infty}\leq
\bar{C}(n)t^{-1}[t|\nabla\chi|^2_{\infty}|_0+1]^{\frac{4}{n}+2}\beta
$
\item $|\chi^2 Ric|_{\infty}\leq
\bar{C}(n,p){A_0}^{\frac{n}{p}}t^{-\frac{n}{2p}}[t|\nabla\chi|^2_{\infty}|_0+1]^{\frac{n}{p}+\frac{2}{p}}K$
\end{enumerate}
\end{thm}

\end{document}